\documentclass[12pt,a4paper]{amsart}

\usepackage{upgreek}
\usepackage{eurosym}
\usepackage{amssymb}
\usepackage{amsfonts}
\usepackage{graphicx}
\usepackage{amsmath}

\usepackage[latin1]{inputenc}
\usepackage{color}
\usepackage{latexsym}
\usepackage{amssymb}
\usepackage{mathrsfs}
\usepackage{enumerate}
\usepackage{amsthm}

\begin{document}
\newtheorem{definition}{Definition}[section]
\newtheorem{example}{Example}[section]
\newtheorem{theorem}{Theorem}[section]
\newtheorem{proposition}{Proposition}[section]
\newtheorem{corollary}{Corollary}[section]
\newtheorem{lemma}{Lemma}[section]
\newtheorem{remark}{Remark}[section]
\newcommand{\coua}{{\hfill\raisebox{1mm}{\framebox[2mm]{}}}}
\newcommand{\ii}\'{\i}
\newcommand{\lid}{\lim_{n\to\infty}}
\newcommand{\parhf}{\par\hfill}
\newcommand{\parni}{\par\noindent}
\newcommand{\cuad}{\marginpar
{\hspace{-8mm}\raisebox{1mm}{\framebox[4mm]{}}}}

\title{Soft convex structures}
\author[Jos\'e Sanabria, Adolfo Pimienta, Semiramis Zambrano]{Jos\'e Sanabria$^1$$^{*}$, Adolfo Pimienta$^2$ and Semiramis Zambrano$^3$}

\address{$^{1}$ Departamento de Matem\'{a}ticas, Facultad de Educaci\'on y Ciencias, Universidad de Sucre, Sincelejo, Colombia.}
\email{jesanabri@gmail.com (Jos\'e E. Sanabria)\newline https://orcid.org/0000-0002-9749-4099}

\address{$^{2}$ Facultad de Ciencias B\'asicas y Biom\'edicas, Universidad Sim\'on Bol\'ivar, Barranquilla, Colombia.}
\email{pimienta331@gmail.com}

\address{$^{3}$  Maestr\'ia en Ciencias Matem\'{a}ticas, Facultad de Ciencias B\'asicas, Universidad del Atl\'antico, Barranquilla, Colombia.}
\email{semy.zambrano@gmail.com}

\date{Received: xxxxxx; Revised: yyyyyy; Accepted: zzzzzz.
\newline \indent $^{*}$Corresponding author}

\begin{abstract}
In this manuscript the idea of soft convex structures is given and some of their properties are investigated. Also, soft convex sets, soft concave sets and soft convex hull operator are defined and their properties are studied. Moreover, the concepts of soft convexly derived operator and soft convex base are studied and their relationship to convex structures are explored.
\end{abstract}

\makeatletter
\@namedef{subjclassname@2020}{%
  \textup{2020} Mathematics Subject Classification}
\makeatother

\maketitle

\section{Preliminaries}
In this section, we present the basic preliminaries strictly necessary for the development of this manuscript; such as the notions of soft set, operations between soft sets and soft function. The concepts given are well known in the literature, therefore, certain citations are opportunely given with respect to the bibliographical references where the notions and results presented here can be investigated. In the development of this paper, $\mathcal{P}(X)$ will denote the power set of a nonempty set $X$.
\begin{definition}\cite{6A}
Let $X$ be an initial universe and $E$ be a non-empty set of parameters. A \emph{soft set} on $X$ is a pair $(\Omega,E)$, where $\Omega:E\to \mathcal{P}(X)$ is a function. In this case, the pair $(\Omega,E)$ is also denoted by $\Omega_{E}$.
\end{definition}
In other words, a soft set on $X$ is a parameterized family of subsets of the universe $X$. For $e\in E$, $\Omega(e)$ can be considered as the set of $e$-approximate elements of the soft set $\Omega_{E}$. The family of all soft sets on $X$ is represented by $\mathcal{S}_{E}(X)$. If $\Omega(e)$ is a finite set for all $e\in E$, then $\Omega_{E}$ is said to be a \textbf{finite soft set}. The family of all finite soft sets on $X$ is denoted by $\mathcal{FS}_{E}(X)$. If $\Omega_{E}$ is not a finite soft set, then $\Omega_{E}$ is said to be a \textbf{non-finite soft set}. Note that $\Omega_{E}$ is a non-finite soft set, if there exists a parameter $e\in E$ such that $\Omega(e)$ is an infinite set.

\begin{definition}\cite{8A}
For $\Omega_{E}, O_{E}\in \mathcal{S}_{E}(X)$, we say that $\Omega_{E}$ is a \textbf{soft subset} of $O_{E}$, if for all $e\in E$, $\Omega(e) \subseteq O(e)$. On the other hand, $\Omega_{E}$ is said to be a \textbf{soft superset} of $O_{E}$, if $O_{E}$ is a soft subset of $\Omega_{E}$. If $\Omega_{E}$ is a soft subset of $O_{E}$, then it is denoted by $\Omega_{E}\widetilde{\subset} O_{E}$.
\end{definition}

\begin{definition}\cite{6A}
Let $\Omega_{E}, O_{E}\in\mathcal{S}_{E}(X)$. We say that $\Omega_{E}$ and $O_{E}$ are \textbf{equal}, if $\Omega_{E}$ is a soft subset of $O_{E}$ and $O_{E}$ is a soft subset of $\Omega_{E}$.
\end{definition}

\begin{definition}\cite{17A}
The \textbf{complement} of a soft set $\Omega_{E}\in \mathcal{S}_{E}(X)$ is denoted by $\Omega_{E}^{c}= (\Omega^{c},E)$, where $\Omega^{c}:E \to \mathcal{P}(X)$ is a function given by $\Omega^{c}(e)=X-\Omega(e) $, for all $e\in E$.
\end{definition}

\begin{definition}\cite{8A}
A soft set $\Omega_{E}\in \mathcal{S}_{E}(X)$ is called a \textbf{null soft set}, denoted by $\Phi_{E}$, if for all $e\in E$, $\Omega(e)=\emptyset$. A soft set that is not null is said to be a \textbf{non-null soft set}.
\end{definition}
Note that for a non-null soft set $\Omega_{E}$, $\Omega(e)\neq \emptyset$ for some $e \in E $.

\begin{definition}\cite{8A}.
A soft set $\Omega_{E}\in \mathcal{S}_{E}(X)$ is called an \textbf{absolute soft set}, denoted by $X_{E}$, if for all $e\in E$, $\Omega(e)=X$.
\end{definition}

\begin{definition}\cite{8A}
The \textbf{union of two soft sets} $\Omega_{E}, O_{E}\in \mathcal{S}_{E}(X)$ is the soft set $\Theta_E=(\Theta,E)$, where for all $e\in E$, $\Theta(e)=\Omega(e) \cup O(e)$. In this case, we write $\Theta_{E}= \Omega_{E}\widetilde{\cup} O_{E}$.
\end{definition}

\begin{definition}\cite{8A}\label{D2.8}
The \textbf{intersection of two soft sets} $\Omega_{E}, O_{E}\in \mathcal{S}_{E}(X)$ is the soft set $\Theta_E=(\Theta,E)$, where for all $e\in E$, $\Theta(e)=\Omega(e) \cap O(e)$. In this case, we write $\Theta_{E}= \Omega_{E}\widetilde{\cap} O_{E}$.
\end{definition}

\begin{definition}\cite{17A}
The \textbf{difference of two soft sets} $\Omega_{E}, O_{E}\in \mathcal{S}_{E}(X)$ is the soft set denoted by $\Omega_{E}\setminus O_{E}=(\Theta,E)$, where for all $e\in E$, $\Theta(e)=\Omega(e)\setminus O(e)$.
\end{definition}

\begin{definition}\cite{8A}
The \textbf{union of a family} $\left\{\left(\Omega_{i}, E\right): i \in I\right\}$ of soft sets on $X$, denoted by $\displaystyle\widetilde{\bigcup}_{i \in I}\left(\Omega_{i}, E\right)$, is the soft set $(\Omega, E)$, where for all $e \in E$, $\Omega(e)=\displaystyle\bigcup_{i \in I} \Omega_{i}(e)$.
\end{definition}

\begin{definition}\cite{8A}
The \textbf{intersection of a family} $\left\{\left(\Omega_{i}, E\right): i \in I\right\}$ of soft sets on $X$, denoted by $\displaystyle\widetilde{\bigcap}_{i \in I}\left(\Omega_{i}, E\right)$, is the soft set $(\Omega, E)$, where for all $e \in E$, $\Omega(e)=\displaystyle\bigcap_{i \in I} \Omega_{i}(e)$.
\end{definition}

\begin{definition}\cite{26A,27A}
A \textbf{soft function} between $\mathcal{S}_{E}\left(X\right)$ and $\mathcal{S}_{E}\left(Y\right)$ is a function $f:X\rightarrow Y$ such that the image of $(\Omega, E) \in\mathcal{S}_{E}\left(X\right)$ and the preimage of $(\Theta, E) \in \mathcal{S}_{E}\left(Y\right)$ are defined by:
\begin{enumerate}[\upshape(1)]
\item $f(\Omega, E)=(f(\Omega), E)$, where $[f(\Omega)](e)=f(\Omega(e))=\{f(x): x\in\Omega(e)\}$, $\forall e \in E$;
\item $f^{-1}(\Theta, E)=\left(f^{-1}(\Theta), E\right)$, where $\left[f^{-1}(\Theta)\right](e)=f^{-1}(\Theta(e))=\{x: f(x)\in\Theta(e)\}$, $\forall e \in E$.
\end{enumerate}
\end{definition}

\section{Soft convex structures and soft convex spaces}
In this section, we introduce and study the notions of soft convex structures and soft convex spaces. Next, we establish some properties related to these notions.
\begin{definition}
A family $\left\{\left(\Omega_{i}, E\right): i \in I\right\}$ of soft sets on $X$ is said to be:
\begin{enumerate}[\upshape(1)]
\item \textbf{Soft downward directed}, if for each pair $i_{1}, i_{2}\in I$, there exists $i_{3}\in I$ such that $\Omega_{i_{3}}\widetilde{\subseteq} \Omega_{i_{1}}$  and $\Omega_{i_{3}}\widetilde{\subseteq} \Omega_{i_{2}}$.
\item \textbf{Soft upward directed}, if for each pair $i_{1}, i_{2}\in I$, there exists $i_{3}\in I$ such that $\Omega_{i_{1}}\widetilde{\subseteq} \Omega_{i_{3}}$  and $\Omega_{i_{2}}\widetilde{\subseteq} \Omega_{i_{3}}$.
\end{enumerate}
\end{definition}

\begin{proposition}\label{P3.1}
Let $\left\{\left(\Omega_{i}, E\right): i \in I\right\}$ be a family of soft sets on $X$. Then, the following properties are equivalent:
\begin{enumerate}[\upshape(1)]
\item $\left\{\left(\Omega_{i}, E\right): i \in I\right\}$ is  soft upward directed.
\item $\left\{\left(\Omega_{i}^{c}, E\right): i \in I\right\}$ is  soft downward directed.
\item $\left\{\Omega_{i}(e): i \in I\right\}$ is  an upward directed family of subsets of $X$, for all $e\in E$.
\item $\left\{\Omega^{c}_{i}(e): i \in I\right\}$ is  a downward directed family of subsets of $X$, for all $e\in E$.
\end{enumerate}
\end{proposition}

\begin{proof}
Straightforward.
\end{proof}

\begin{definition}\label{D3.2}
A family $\zeta\subseteq \mathcal{S}_{E}(X)$ is called a \textbf{soft convex structure} on $X$ if $\zeta$ satisfies the following conditions:
\begin{enumerate}[\upshape(1)]
\item $\Phi_{E}$, $X_{E}$ belong to $\zeta$.
\item The intersection of each family of members of $\zeta$ belongs to $\zeta$.
\item The union of each soft upward directed family of members of $\zeta$ belongs to $\zeta$.
\end{enumerate}
\end{definition}

The triplet $(X,\zeta, E)$ is called a \textbf{soft convex space}.

\begin{definition}\label{D3.3}
Let $(X,\zeta, E)$ be a soft convex space. The members of $\zeta$ are called \textbf{soft convex sets} on $X$. A soft set $\Omega_{E}$ on $X$ is said to be a \textbf{soft concave set} on $X$, if $\Omega^{c}_{E}$ belongs to $\zeta$.
\end{definition}

\begin{proposition}\label{P3.2}
Let $(X,\zeta, E)$ be a soft convex space. Then, the following properties hold:
\begin{enumerate}[\upshape(1)]
\item $\Phi_{E}$, $X_{E}$ are soft concave sets on $X$.
\item The union of each family of soft concave sets on $X$ is a soft concave set on $X$.
\item The union of every soft downward directed family of soft concave sets on $X$ is a soft concave set on $X$.
\end{enumerate}
\end{proposition}

\begin{proof}
Follows from Definitions \ref{D3.2} and \ref{D3.3} together with Proposition \ref{P3.1}.
\end{proof}

\begin{proposition}\label{P3.2}
Let $(X, \zeta, E)$ be a soft convex space. Then, the family
$\zeta(e)=\{\Omega(e): (\Omega,E)\in \zeta\}$ for all $e \in E$, defines a convex structure on $X$  called the \textbf{$e$-parameter convex structure} (or \textbf{crisp convex structure}).
\end{proposition}

\begin{proof}
By definition, for each $e \in E$, we have $\zeta(e)=\{\Omega(e):(\Omega,E)\in\zeta\}$.
Now, we verify that the family $\zeta(e)=\{\Omega(e):(\Omega,E)\in\zeta\}$ is a convex structure on $X$, for all $e \in E$. Indeed:\\
(1) Since $\Phi_{E}, X_{E}\in \zeta$, it follows that $\emptyset, X \in \zeta(e)$, for all $e \in E$.\\
(2) Let $\left\{\Omega_{i}(e): i \in I\right\}$ be a family of sets in $\zeta(e)$, for all $e\in E$. Since $(\Omega_{i},E) \in \zeta$, for all $i \in I$,  we have $\displaystyle\widetilde{\bigcap}_{i \in I}\left(\Omega_{i}, E\right) \in \zeta$, which implies that $\displaystyle\bigcap_{i \in I} \Omega_{i}(e) \in \zeta(e)$, for all $e\in E$.\\
(3) Let $\left\{\Omega_{i}(e): i \in I\right\}$ be an upward directed family of sets in $\zeta(e)$, for all $e\in E$. By Proposition \ref{P3.1}, $\{(\Omega_{i},E): i\in I\}$ is a soft upward directed family and as $(\Omega_{i},E) \in \zeta$, for all $i \in I$,  we have $\displaystyle\widetilde{\bigcup}_{i \in I}\left(\Omega_{i}, E\right) \in \zeta$. Thus, $\displaystyle\bigcup_{i \in I} \Omega_{i}(e) \in \zeta(e)$, for all $e\in E$.

This shows that $\zeta(e)$ defines a convex structure on $X$ for all $e \in E$.
\end{proof}

\begin{remark}
Proposition \ref{P3.2} tells us that for each parameter $e\in E$, one has a convex structure $\zeta(e)$ on $X$. According to this, a soft convex structure on $X$ induces a parameterized family of convex structures on $X$.
\end{remark}

The following example shows that the converse of Proposition \ref{P3.2}, in general, is not true.
\begin{example} Let $X=\left\{x_1, x_2, x_3\right\}$ and $E=\left\{e_1, e_2\right\}$. Consider the following soft sets on $X$:
$$
\begin{aligned}
& \left(\Omega_1, E\right)=\left\{\left(e_1,\left\{x_1\right\}\right),\left(e_2,\left\{x_1\right\}\right)\right\},\\
& \left(\Omega_2, E\right)=\left\{\left(e_1,\left\{x_1\right\}\right),\left(e_2,\left\{x_1, x_2\right\}\right)\right\},\\
& \left(\Omega_3, E\right)=\left\{\left(e_1, X\right),\left(e_2,\left\{x_1, x_3\right\}\right)\right\},\\
& \left(\Omega_4, E\right)=\left\{\left(e_1,\left\{x_2\right\}\right),\left(e_2,\left\{x_2\right\}\right)\right\},\\
& \left(\Omega_5, E\right)=\left\{\left(e_1,\left\{x_1, x_2\right\}\right),\left(e_2,\left\{x_1, x_2\right\}\right)\right\}.
\end{aligned}
$$
Put
$$
\begin{aligned}
\zeta= & \left\{\Phi_E,\left(\Omega_1, E\right),\left(\Omega_2, E\right),\left(\Omega_3, E\right),\left(\Omega_4, E\right), \left(\Omega_5, E\right),  X_E\right\}.
\end{aligned}
$$
Then,
$$
\begin{aligned}
 \zeta\left(e_1\right)&=\left\{\Phi_E(e_{1}),\Omega_1\left(e_{1}\right), \Omega_2\left(e_{1}\right), \Omega_3\left(e_{1}\right), \Omega_4\left(e_{1}\right), \Omega_5\left(e_{1}\right),X_{E}(e_{1})\right\}\\
&=\left\{\emptyset,\left\{x_1\right\},\left\{x_2\right\},\left\{x_1, x_2\right\}, X\right\}
\end{aligned}
$$
and
$$
\begin{aligned}
\zeta\left(e_2\right)&=\left\{\Phi_E(e_{2}),\Omega_1\left(e_{2}\right), \Omega_2\left(e_{2}\right), \Omega_3\left(e_{2}\right), \Omega_4\left(e_{2}\right), \Omega_5\left(e_{2}\right),X_{E}(e_{2})\right\}\\
&=\left\{\emptyset,\left\{x_1\right\},\left\{x_2\right\},\left\{x_1, x_2\right\},\left\{x_1, x_3\right\}, X\right\}
\end{aligned}
$$
are convex structures on $X$, while $\zeta$ is not a soft convex structure on $X$, because
$$
\left(\Omega_2, E\right)\widetilde{\cap}\left(\Omega_4, E\right)=\left\{\left(e_1, \emptyset\right),\left(e_2,\left\{x_2\right\}\right)\right\} \notin \zeta.
$$
\end{example}

In the previous example we can notice that any family of soft sets is not necessarily a soft convex structure on $X$, even if the family corresponding to each parameter is a convex structure on $X$. When the parameter set $E$ is a singleton the situation changes, as we see in the following result.
\begin{proposition}\label{P3.3}
Let $X$ be any set, $E=\{e\}$, $\zeta=\{(e, \Omega(e)): (\Omega,E)\in \mathcal{S}_{E}(X)\}$ and $\zeta(e)=\{\Omega(e):(e, \Omega(e)) \in \zeta\}$. Then, $\zeta$ is a soft convex structure on $X$ if and only if $\zeta(e)$ is a convex structure on $X$.
\end{proposition}

\begin{proof}
It suffices to show that if $\zeta(e)$ is a convex structure on $X$, then $\zeta$  is a soft convex structure. By assumption, $\emptyset, X \in \zeta(e)$, which implies that $(e, \emptyset) \in\zeta$ and $(e, X) \in\zeta$ and so, $\Phi_E, X_E \in \zeta$. Let $(\Omega, E)=(e,\Omega(e))$  and $(O, E)=(e,O(e)) \in \zeta$. Then, $\Omega(e),O(e) \in \zeta(e)$ and as $\zeta(e)$ is a convex structure on $X$, it follows that $\Omega(e) \cap O(e) \in\zeta(e)$.
Therefore, $(\Omega, E) \widetilde{\cap}(O, E)=(e, \Omega(e)) \widetilde{\cap}(e, O(e))=(e, \Omega(e) \cap O(e)) \in \zeta$. Suppose that $\{\left(\Omega_{i}, E\right):i\in I\}$ is a soft upward directed family, where $\left(\Omega_{i}, E\right)=\left(e, \Omega_{i}(e)\right)\in \zeta$, for all $i\in I$. Then, $\{\Omega_{i}(e):i\in I\}$ is an upward directed family of sets belonging to $\zeta(e)$ and as $\zeta(e)$ is a convex structure, it follows that $\displaystyle\bigcup_{i \in I} \Omega_{i}(e)\in\zeta(e)$. Therefore, $\displaystyle\widetilde{\bigcup}_{i\in I}\left(\Omega_{i}, E\right)=\widetilde{\bigcup}_{i\in I}\left(e, \Omega_{i}(e)\right)=\left(e, \bigcup_{i \in I} \Omega_{i}(e)\right) \in \zeta$, which completes the proof.
\end{proof}

\begin{proposition}\label{P3.5}
If \, $\Upsilon$ is a convex structure on a set $X$, then the family of soft sets
\begin{center}
$\zeta(\Upsilon)=\{(\Omega,E)\in \mathcal{S}_{E}(X):\Omega(e)\in\Upsilon \mbox{ for all } e\in E\}$
\end{center}
is a soft convex structure on $X$.
\end{proposition}

\begin{proof}
Straightforward.
\end{proof}
We will call $\zeta(\Upsilon)$ the \textbf{soft convex structure on $X$ induced by $\Upsilon$}. If each parameter $e \in E$ is assigned to the same set $A \in \Upsilon$; this is, $\Omega(e)=A, \, \forall e\in E$, then the obtained soft convex structure we will call the \textbf{single-set soft convex structure on $X$ induced by $\Upsilon$} (denoted by $\widetilde{\zeta}(\Upsilon)$). Note that for each $e\in E$, we have $\zeta(e)=\widetilde{\zeta}(e)=\Upsilon$.\\

\section{Soft convex hull operator}
In this part, we present the concept of soft convex hull of a soft set and study its relationship with other notions defined in the context of a soft convex space.
\begin{definition}\label{Def4.1}
Let $(X,\zeta, E)$ be a soft convex space and $\Omega_{E}\in\mathcal{S}_{E}(X)$. The intersection of all the soft convex supersets of $\Omega_{E}$, denoted by $co(\Omega_{E})$ is called the \textbf{soft convex hull}  of $\Omega_{E}$; this is,
\begin{center}
$co(\Omega_{E})=\displaystyle\widetilde{\bigcap}\{O_{E}: \Omega_{E}\widetilde{\subseteq} O_{E},\, O_{E}\in\zeta\}$.
\end{center}
\end{definition}
Note that $co(\Omega_{E})$ is the smallest element (in the sense of the soft inclusion) of $\zeta$ that contains $\Omega_{E}$; i.e., if $\Theta_{E}\in\zeta$ and $\Omega_{E}\widetilde{\subseteq} \Theta_{E}$, then $co(\Omega_{E})\widetilde{\subseteq} \Theta_{E}$. The soft operator $co$ is called the \textbf{soft convex hull operator} on $(X,\zeta,E)$.

\begin{proposition}
Let $co$ be soft hull operator in a soft convex space $(X,\zeta,E)$ and $\Omega_{E}, O_{E}\in\mathcal{S}_{E}(X)$. Then, the following properties hold:
\begin{enumerate}[\upshape(1)]
\item $co(\Phi_{E})=\Phi_{E}$ (\textbf{Soft normalization law}),
\item $\Omega_{E}\widetilde{\subseteq} co(\Omega_{E})$ (\textbf{Soft extensive law}),
\item If $\Omega_{E}\widetilde{\subseteq} O_{E}$, then $co(\Omega_{E})\widetilde{\subseteq} co(O_{E})$ (\textbf{Soft monotone law}).
\item $co(co(\Omega_{E}))=co(\Omega_{E})$ (\textbf{Soft idempotent law}),
\item If $\{(\Omega_{i},E): i\in I\}\subseteq \mathcal{S}_{E}(X)$ is a soft upward directed family, then\\
$co\left(\widetilde{\bigcup}_{i\in I} (\Omega_{i},E)\right)=\displaystyle\widetilde{\bigcup}_{i\in I}co(\Omega_{i},E)$ (\textbf{Soft upward directed additive law}).
\item $\Omega_{E}\in\zeta$ if and only if $\Omega_{E}=co(\Omega_{E})$.
\end{enumerate}
\end{proposition}

\begin{proof}
(1) By soft set theory, it always holds that $\Phi_{E}\widetilde{\subseteq} co(\Phi_{E})$. On the other hand, since $\Phi_{E}\widetilde{\subseteq} \Phi_{E}$ and $\Phi_{E}\in\zeta$, it follows that $co(\Phi_{E})\widetilde{\subseteq}\Phi_{E}$.\\
(2)  This is an immediate consequence of the definition of $co(\Omega_{E})$.\\
(3) Suppose that $\Omega_{E}\widetilde{\subseteq} O_{E}$. Then, $\Omega_{E}\widetilde{\subseteq} O_{E}\widetilde{\subseteq} co(O_{E})$ and as $co(O_{E})\in\zeta$, it follows that $co(\Omega_{E})\widetilde{\subseteq} co(O_{E})$.\\
(4) By the soft extensive law, $co(\Omega_{E})\widetilde{\subseteq} co(co(\Omega_{E}))$. On the other hand, since $co(\Omega_{E})\widetilde{\subseteq} co(\Omega_{E})$ and $co(\Omega_{E})\in\zeta$, it follows that $co(co(\Omega_{E}))\widetilde{\subseteq} co(\Omega_{E})$.\\
(5) Suppose that $\{(\Omega_{i},E): i\in I\}\subseteq \mathcal{S}_{E}(X)$  is a soft upward directed family. By the soft monotone law, $\{co(\Omega_{i},E): i\in I\}$ is also a soft upward directed family and since $\{co(\Omega_{i},E): i\in I\}\subseteq \zeta$, we have $\displaystyle\widetilde{\bigcup}_{i\in I}co(\Omega_{i},E)\in\zeta$. Now, by the soft extensive law, $\displaystyle\widetilde{\bigcup}_{i\in I}(\Omega_{i},E)\widetilde{\subseteq} \displaystyle\widetilde{\bigcup}_{i\in I}co(\Omega_{i},E)\in\zeta$, which implies that $co\left(\displaystyle\widetilde{\bigcup}_{i\in I}(\Omega_{i},E)\right)\widetilde{\subseteq} \displaystyle\widetilde{\bigcup}_{i\in I}co(\Omega_{i},E)$. The opposite soft inclusion is obviously obtained by the soft monotone law.\\
(6) By the soft extensive law, we always have $\Omega_{E}\widetilde{\subseteq} co(\Omega_{E})$. For the opposite inclusion, suppose that $\Omega_{E}\in\zeta$. Since $\Omega_{E}\widetilde{\subseteq} \Omega_{E}$ and $\Omega_{E}\in\zeta$, we obtain that $co(\Omega_{E})\widetilde{\subseteq} O_{E}$. Conversely, suppose that $\Omega_{E}=co(\Omega_{E})$. Since $co(\Omega_{E})\in\zeta$, the proof is finished.
\end{proof}

\begin{theorem}\label{Theo4.1}
Let $\eta:\mathcal{S}_{E}(X)\to\mathcal{S}_{E}(X)$ be a mapping that satisfies the following properties:
\begin{enumerate}[\upshape(1)]
\item $\eta(\Phi_{E})=\Phi_{E}$ (\textbf{Soft normalization law}),
\item $\Omega_{E}\widetilde{\subseteq} \eta(\Omega_{E})$ (\textbf{Soft extensive law}),
\item If $\Omega_{E}\widetilde{\subseteq} O_{E}$, then $\eta(\Omega_{E})\widetilde{\subseteq} \eta(O_{E})$ (\textbf{Soft monotone law}),
\item $\eta(\eta(\Omega_{E}))=\eta(\Omega_{E})$ (\textbf{Soft idempotent law}),
\item If $\{(\Omega_{i},E): i\in I\}\subseteq \mathcal{S}_{E}(X)$ is a soft upward directed family, then\\
$\eta\left(\widetilde{\bigcup}_{i\in I} (\Omega_{i},E)\right)=\displaystyle\widetilde{\bigcup}_{i\in I}\eta(\Omega_{i},E)$ (\textbf{Soft upward directed additive law}).
\end{enumerate}
Then, the family
\begin{center}
$\zeta_{\eta}=\{\Omega_{E}\in\mathcal{S}_{E}(X):\eta(\Omega_{E})=\Omega_{E}\}$
\end{center}
is a soft convex structure on $X$ such that $\eta(\Omega_{E})=co(\Omega_{E})$ for all $\Omega_{E}\in\mathcal{S}_{E}(X)$.
\end{theorem}

\begin{proof}
We will check that the family $\zeta_{\eta}=\{\Omega_{E}\in\mathcal{S}_{E}(X):\eta(\Omega_{E})=\Omega_{E}\}$ is a soft convex structure on $X$:\\
(1) Obviously $\Phi_{E}\in\zeta_{\eta}$, because $\Phi_{E}\in\mathcal{S}_{E}(X)$ and $\eta(\Phi_{E})= \Phi_{E}$ by the soft normalization law. Since $\eta(X_{E})\widetilde{\subseteq} X_{E}$ and $X_{E}\widetilde{\subseteq} \eta(X_{E})$ by the soft extensive law, we have $\eta(X_{E})=X_{E}$. Therefore, also $X_{E}\in \zeta_{\eta}$.\\
(2) Suppose that $\{(\Omega_{i},E): i\in I\}\subseteq\zeta_{\eta}$ and let $(\Theta,E)=\displaystyle\widetilde{\bigcap}_{i\in I} (\Omega_{i},E)$. Then, $(\Omega_{i},E)\in \mathcal{S}_{E}(X)$ and $\eta(\Omega_{i},E)= (\Omega_{i},E)$, for all $i\in I$. Observe that  $(\Theta,E)=\displaystyle\widetilde{\bigcap}_{i\in I} (\Omega_{i},E)\widetilde{\subseteq} (\Omega_{j},E)$ for all $j\in I$. By the soft monotone law, we have $\eta(\Theta,E)\widetilde{\subseteq} \eta(\Omega_{j},E)=(\Omega_{j},E)$ for all $j\in I$. Thus, $\eta(\Theta,E)\widetilde{\subseteq} \displaystyle\widetilde{\bigcap}_{i\in I} (\Omega_{i},E)=(\Theta,E)$. By the soft extensive law, we obtain that $(\Theta,E)\widetilde{\subseteq} \eta(\Theta,E)$. Therefore,
$\eta(\Theta,E)= (\Theta,E)$ and so, $(\Theta,E)\in \zeta_{\eta}$.\\
(3) Suppose that $\{(\Omega_{i},E): i\in I\}\subseteq\zeta_{\eta}$ is a soft upward directed family. Then, $(\Omega_{i},E)\in \mathcal{S}_{E}(X)$ and $\eta(\Omega_{i},E)= (\Omega_{i},E)$ for all $i\in I$. Now, as $\displaystyle\widetilde{\bigcup}_{i\in I} (\Omega_{i},E)\in \mathcal{S}_{E}(X)$, by the soft upward directed additive law, we get that $\eta\left(\displaystyle\widetilde{\bigcup}_{i\in I} (\Omega_{i},E)\right)=\displaystyle\widetilde{\bigcup}_{i\in I} \eta(\Omega_{i},E)=\displaystyle\widetilde{\bigcup}_{i\in I} (\Omega_{i},E)$ and hence, $\displaystyle\widetilde{\bigcup}_{i\in I} (\Omega_{i},E)\in\zeta_{\eta}$.

Certainly $\zeta_{\eta}$ is a soft convex structure on $X$. Now, we will show that with respect to this, $\eta(\Omega_{E})=co(\Omega_{E})$ for all $\Omega_{E}\in \mathcal{S}_{E}(X)$. Note that $\eta(\Omega_{E})\in \zeta_{\eta}$, because by the soft idempotent law, $\eta(\eta(\Omega_{E}))=\eta(\Omega_{E})$. Moreover, by the soft extensive law, $\Omega_{E}\widetilde{\subseteq}\eta(\Omega_{E})$, and since $\eta(\Omega_{E})\in\zeta_{\eta}$, it follows that $co(\Omega_{E})\widetilde{\subseteq} \eta(\Omega_{E})$. On the other hand, given that $co(\Omega_{E})\in\zeta_{\eta}$  and $\Omega_{E}\widetilde{\subseteq} co(\Omega_{E})$, we have $\eta(\Omega_{E})\widetilde{\subseteq} \eta(co(\Omega_{E}))=co(\Omega_{E})$. Therefore, $\eta(\Omega_{E})=co(\Omega_{E})$ whenever $\Omega_{E}\in \mathcal{S}_{E}(X)$.
\end{proof}

We will call a \textbf{soft hull operator} to any mapping $\eta:\mathcal{S}_{E}(X)\to \mathcal{S}_{E}(X)$ that satisfies properties (1)-(5) of Theorem \ref{Theo4.1}.

\begin{definition} Let $(X,\zeta,E)$ be a soft convex space and $(\Omega,E)\in\mathcal{S}_{E}(X)$. Then, we associate with $(\Omega,E)$ a soft set on $X$, denoted by $(co\Omega,E)$ and defined as
$co\Omega(e)=co(\Omega(e))$, where $co(\Omega(e))$ is the convex hull of $\Omega(e)$ in $\zeta(e)$ for all $e\in E$.
\end{definition}

\begin{proposition}\label{Prop4.2}
Let $(X,\zeta,E)$ be a soft convex space. Then, $(co\Omega,E)\widetilde{\subseteq}co(\Omega,E)$ for all $(\Omega,E)\in\mathcal{S}_{E}(X)$.
\end{proposition}

\begin{proof}
For each $e\in E$, we have  $co(\Omega(e))$ is the smallest soft $\zeta(e)$-convex set that contains $\Omega(e)$. Put $(\Theta,E)=co(\Omega,E)$. Then, $\Theta(e)$ is also a soft $\zeta(e)$-convex set containing $\Omega(e)$ for all $e\in E$. Thus, $co\Omega(e)=co(\Omega(e)) \widetilde{\subseteq} \Theta(e)$ for all $e\in E$. Therefore, $(co\Omega, E) \widetilde{\subseteq} co(\Omega, E)$.
\end{proof}

\begin{corollary}
Let $(X, \zeta, E)$ be a soft convex space and $(\Omega,E)\in\mathcal{S}_{E}(X)$. Then, $(co\Omega, E)=co(\Omega, E)$ if and only if $(co\Omega^{c},E)$ is a soft $\zeta$-concave set.
\end{corollary}

\begin{proof}
If $(co\Omega, E)=co(\Omega, E)$, then $(co\Omega, E)$ is a soft convex set and hence, $(co\Omega^{c}, E)$ is a soft $\zeta$-concave set. Conversely, if $(co\Omega^{c}, E)$ is a soft $\zeta$-concave set, then $(co\Omega, E)$ is a soft convex set containing $(\Omega, E)$. By Proposition \ref{Prop4.2}, $(co\Omega,E)\widetilde{\subseteq}co(\Omega,E)$ and by the definition of soft convex hull of $(\Omega, E)$, each soft convex set on $X$ that contains $(\Omega, E)$ will contain $co(\Omega, E)$. Thus, $co(\Omega, E) \widetilde{\subseteq}(co\Omega, E)$ and hence, $(co\Omega, E)=co(\Omega, E)$.
\end{proof}

\begin{example}
Consider the soft convex space $(X,\zeta,E)$ given in Example \ref{Exa3.4} and the soft set $(\Omega,E)=\left\{\left(e_1,\left\{x_{3}\right\}\right),\left(e_2,\left\{x_{1},x_{3}\right\}\right)\right\}$. Then,
$co(\Omega,E)=X_{E}$ and as $co(\Omega(e_{1}))=\left\{x_3\right\}$ and $co(\Omega(e_{2}))=\left\{x_1, x_3\right\}$, we have $(co\Omega,E)=\left\{\left(e_1,\left\{x_3\right\}\right),\left(e_2,\left\{x_1, x_3\right\}\right)\right\}$. Observe that $(co\Omega,E)\widetilde{\subseteq} co(\Omega,E)$, but $(co\Omega,E)\neq co(\Omega,E)$.
\end{example}

\begin{definition}
Let $(X,\zeta_{X},E)$ and $(Y,\zeta_{Y},E)$ be two soft convex spaces. A soft function $f:X\to Y$  is called:
\begin{enumerate}[\upshape(1)]
\item \textbf{Soft convex preserving} (briefly \textbf{SCP}), if $O_{E}\in\zeta_Y$ implies  $f^{-1} (O_{E})\in\zeta_{X}$.
\item \textbf{Soft convex to convex} (briefly \textbf{SCC}), if $\Omega_{E}\in\zeta_X$  implies $f(\Omega_{E})\in\zeta_{Y}$.
	\end{enumerate}
\end{definition}

We will use the notation $f:(X,\zeta_{X},E)\to (Y,\zeta_{Y},E)$ to represent a soft function between two soft convex spaces $(X,\zeta_{X},E)$ and $(Y,\zeta_{Y},E)$.

Henceforth, given $\Omega_{E}\in \mathcal{S}_{E}(X)$, put $\mathcal{P}(\Omega_{E})=\{O_{E}\in \mathcal{S}_{E}(X): O_{E}\widetilde{\subseteq} \Omega_{E}\}$.
\begin{theorem}\label{teorema 4.2}
Let $f:(X,\zeta_{X},E)\to (Y,\zeta_{Y},E)$  be a soft function between two soft convex spaces. Then, the following properties are equivalent:
\begin{enumerate}[\upshape(1)]
\item $f$ is SCP.
\item If $\{(\Omega_{i},E): i\in I\}\subseteq \mathcal{S}_{E}(X)$ is a soft upward directed family, then\\
$f\left(co_{X} \left(\displaystyle\widetilde{\bigcup}_{i\in I} (\Omega_{i},E)\right)\right)\widetilde{\subseteq} \displaystyle\widetilde{\bigcup}_{i\in I} co_{Y}(f(\Omega_{i},E))$.
\item $f(co_{X}(\Omega_{E}))\widetilde{\subseteq} co_{Y}(f(\Omega_{E}))$ for all $\Omega_{E}\in\mathcal{S}_{E}(X)$.
\end{enumerate}
\end{theorem}

\begin{proof}
(1) $\Longrightarrow$ (3). Suppose that $\Omega_{E}\in\mathcal{S}_{E}(X)$ and let $\Theta_{E}\in \zeta_{Y}$ such that $f(\Omega_{E})\widetilde{\subseteq} \Theta_{E}$. Since $f$ is SCP, we have $f^{-1}(\Theta_{E})\in \zeta_{X}$ and as $\Omega_{E}\widetilde{\subseteq} f^{-1}(f(\Omega_{E}))\widetilde{\subseteq} f^{-1}(\Theta_{E})$, it follows that $co_{X}(\Omega_{E})\widetilde{\subseteq} f^{-1}(\Theta_{E})$. Thus, $f(co_{X}(\Omega_{E}))\widetilde{\subseteq} f(f^{-1}(\Theta_{E}))\widetilde{\subseteq} \Theta_{E}$. Since $\Theta_{E}$ is an arbitrary soft $\zeta_{Y}$-convex superset of $f(co_{X}(\Omega_{E}))$, we get that $f(co_{X}(\Omega_{E}))\widetilde{\subseteq} \displaystyle\widetilde{\bigcap}\left\{\Theta_{E}\in\zeta_{Y}: f(\Omega_{E})\widetilde{\subseteq} \Theta_{E}\right\}=co_{Y}(f(\Omega_{E}))$.

(3) $\Longrightarrow$ (1). Let $O_{E}\in \zeta_{Y}$. Then, $f^{-1}(O_{E})\in
\mathcal{S}_{E}(X)$. By the hypothesis and the soft monotone law, we have $f(co_{X}(f^{-1}(O_{E})))\widetilde{\subseteq} co_{Y}(f(f^{-1}(O_{E})))\widetilde{\subseteq} co_{Y}(O_{E})=O_{E}$ and so,
$co_{X}(f^{-1}(O_{E}))\widetilde{\subseteq} f^{-1}(f(co_{X}(f^{-1}(O_{E}))))\widetilde{\subseteq} f^{-1}(O_{E})$. On the other hand, since the soft inclusion $f^{-1}(O_{E})\widetilde{\subseteq} co_{X}(f^{-1}(O_{E}))$ is always true, we conclude that $co_{X}(f^{-1}(O_{E}))= f^{-1}(O_{E})$ and hence, $f^{-1}(O_{E})\in \zeta_{X}$,  which proves that $f$ is SCP.

(2) $\Longrightarrow$ (3).
Let $\Omega_{E}\in \mathcal{S}_{E}(X)$. Using the fact that $\{f\left(O_{E}\right):O_{E}\in\mathcal{P}(\Omega_{E})\}$ is a soft upward directed family, we have
\begin{eqnarray*}
f(co_{X}(\Omega_{E}))&=& f\left(co_{X}\left(\displaystyle\widetilde{\bigcup}\{O_{E}:O_{E}\in\mathcal{P}(\Omega_{E})\}\right)\right)\\
&\widetilde{\subseteq} & \displaystyle\widetilde{\bigcup}\{co_{Y}(f\left(O_{E})\right): O_{E}\in\mathcal{P}(\Omega_{E})\}\\
&= &co_{Y}\left(\displaystyle\widetilde{\bigcup}\{f(O_{E}): O_{E}\in\mathcal{P}(\Omega_{E})\}\right)\\
&= &co_{Y}\left(f\left(\displaystyle\widetilde{\bigcup}\{O_{E}: O_{E}\in\mathcal{P}(\Omega_{E})\}\right)\right)\\
&=& co_{Y}(f(\Omega_{E})).
\end{eqnarray*}

(3) $\Longrightarrow$ (2). Let $\{(\Omega_{i},E): i\in I\}\subseteq \mathcal{S}_{E}(X)$ be a soft upward directed family. Then,
\begin{eqnarray*}
f\left(co_{X} \left(\displaystyle\widetilde{\bigcup}_{i\in I} (\Omega_{i},E)\right)\right)&=& f\left(\displaystyle\widetilde{\bigcup}_{i\in I}co_{X} \left( \Omega_{i},E\right)\right)\\
&= & \displaystyle\widetilde{\bigcup}_{i\in I} f\left(co_{X} \left( \Omega_{i},E\right)\right)\\
&\widetilde{\subseteq} & \displaystyle\widetilde{\bigcup}_{i\in I} co_{Y} \left(f\left(\Omega_{i},E\right)\right).
\end{eqnarray*}
\end{proof}

\begin{proposition}
Let $f: (X,\zeta_{X})\to(Y,\zeta_{Y})$ and $g:(Y,\zeta_{Y}) \to(Z,\zeta_{Z})$ be two soft functions between soft convex spaces. If $f$ and $g$ are SCP, then  $g \circ f$  is also SCP.
\end{proposition}

\begin{theorem}\label{Theo4.3}
Let $f:(X,\zeta_{X},E)\to (Y,\zeta_{Y},E)$  be a soft function between two soft convex spaces. Then, the following properties are equivalent:
\begin{enumerate}[\upshape(1)]
\item $f$ is SCC.
\item If $\{(\Omega_{i},E): i\in I\}\subseteq \mathcal{S}_{E}(X)$ is a soft upward directed family, then\\
$\displaystyle\widetilde{\bigcup}_{i\in I} co_{Y}(f(\Omega_{i},E))\widetilde{\subseteq} f\left(co_{X} \left(\displaystyle\widetilde{\bigcup}_{i\in I} (\Omega_{i},E)\right)\right)$.
\item $co_{Y}(f(\Omega_{E}))\widetilde{\subseteq} f(co_{X}(\Omega_{E}))$ for all $\Omega_{E}\in\mathcal{S}_{E}(X)$.
\end{enumerate}
\end{theorem}

\begin{proof}
(1) $\Longrightarrow$ (3). Let $\Omega_{E}\in\mathcal{S}_{E}(X)$. Then,
$co_{X}(\Omega_{E})\in \zeta_{X}$
and as $f$ is SCC, we have $f(co_{X}(\Omega_{E}))\in \zeta_{Y}$, which implies that
$co_{Y}(f(co_{X}(\Omega_{E})))\widetilde{\subseteq} f(co_{X}(\Omega_{E}))$. By the soft extensive and soft monotone laws, we conclude that
$co_{Y}(f(\Omega_{E}))\widetilde{\subseteq} co_{Y}(f(co_{X}(\Omega_{E})))\widetilde{\subseteq} f(co_{X}(\Omega_{E}))$.

(3) $\Longrightarrow$ (1). Let $\Omega_{E}\in \zeta_{X}$. Then, $\Omega_{E}\in\mathcal{S}_{E}(X)$ and $co_{X}(\Omega_{E})=\Omega_{E}$. By hypothesis, we have
$co_{Y}(f(\Omega_{E}))\widetilde{\subseteq} f(co_{X}(\Omega_{E}))=f(\Omega_{E})$ and so,
$f(\Omega_{E})\in \zeta_{Y}$, which shows that $f$ is SCC.

(2) $\Longrightarrow$ (3).
Let $\Omega_{E}\in \mathcal{S}_{E}(X)$. Since $\{f\left(O_{E}\right):O_{E}\in\mathcal{P}(\Omega_{E})\}$ is a soft upward directed family, we have
\begin{eqnarray*}
co_{X}(f(\Omega_{E}))&=& co_{X}\left(f\left(\displaystyle\widetilde{\bigcup}\{O_{E}:O_{E}\in\mathcal{P}(\Omega_{E})\}\right)\right)\\
&= & co_{X}\left(\displaystyle\widetilde{\bigcup}\{f\left(O_{E}\right):O_{E}\in\mathcal{P}(\Omega_{E})\}\right)\\
&= & \displaystyle\widetilde{\bigcup}\{co_{X}\left(f\left(O_{E}\right)\right):O_{E}\in\mathcal{P}(\Omega_{E})\}\\
&\widetilde{\subseteq}& f\left(co_{Y}\left(\displaystyle\widetilde{\bigcup}\{O_{E}:O_{E}\in\mathcal{P}(\Omega_{E})\}\right)\right)\\
&=& f(co_{Y}(\Omega_{E})).
\end{eqnarray*}

(3) $\Longrightarrow$ (2). Let $\{(\Omega_{i},E): i\in I\}\subseteq \mathcal{S}_{E}(X)$ be a soft upward directed family. Then,
\begin{eqnarray*}
\displaystyle\widetilde{\bigcup}_{i\in I} co_{Y} \left(f\left(\Omega_{i},E\right)\right)&\widetilde{\subseteq}& \displaystyle\widetilde{\bigcup}_{i\in I} f\left(co_{X} \left( \Omega_{i},E\right)\right)\\
&= & f\left(\displaystyle\widetilde{\bigcup}_{i\in I}co_{X} \left( \Omega_{i},E\right)\right)\\
&= & f\left(co_{X} \left(\displaystyle\widetilde{\bigcup}_{i\in I} (\Omega_{i},E)\right)\right).
\end{eqnarray*}
\end{proof}

\begin{proposition}
Let $f: (X,\zeta_{X})\to(Y,\zeta_{Y})$ and $g:(Y,\zeta_{Y}) \to(Z,\zeta_{Z})$ be two soft functions between soft convex spaces. If $f$ and $g$ are SCC, then  $g \circ f$  is also SCC.
\end{proposition}

\section{Soft c-derived operators and soft c-derived spaces}
In this section, we introduce the notions of soft convexly derived operators and soft convexly derived spaces. In addition, we describe the relationship between the above notions and soft convex structures.

\begin{definition}\label{Def4.1}
A mapping $d:\mathcal{S}_{E}(X)\to\mathcal{S}_{E}(X)$ is called a \textbf{soft convexly derived operator} (briefly \textbf{soft c-derived operator}) on $X$, if the following conditions are satisfied:
\begin{enumerate}[\upshape(1)]
\item $d(\Phi_{E})=\Phi_{E}$ (\textbf{Soft normalization law}),
\item  $\Omega_{E}\widetilde{\subseteq} O_{E}$ implies $d(\Omega_{E})\widetilde{\subseteq} d(O_{E})$ (\textbf{Soft monotone law}),
\item $d(d(\Omega_{E})\cup \Omega_{E})\widetilde{\subseteq} d(\Omega_{E})\cup \Omega_{E}$ (\textbf{Soft idempotent law}),
\item If $\{(\Omega_{i},E): i\in I\}\subseteq \mathcal{S}_{E}(X)$ is a soft upward directed family, then\\
$d\left(\widetilde{\bigcup}_{i\in I} (\Omega_{i},E)\right)=\displaystyle\widetilde{\bigcup}_{i\in I} d(\Omega_{i},E)$ (\textbf{Soft upward directed additive law}).
\end{enumerate}
\end{definition}

If $d$ is a soft convexly derived operator on $X$, then the triplet $(X,d,E)$ is called a \textbf{soft convexly derived space} (briefly \textbf{soft c-derived space}).

\begin{proposition}\label{prop 4.3.6}
Let $d$ be a soft c-derived operator on $X$. Then, the family $\zeta_{d}\subseteq\mathcal{S}_{E}(X)$ defined by
\begin{center}
$\zeta_{d}=\{\Omega_{E}\in\mathcal{S}_{E}(X): d(\Omega_{E})\widetilde{\subseteq} \Omega_{E}\}$,
\end{center}
is a soft convex structure on $X$, called the \textbf{soft convex structure induced by $d$}. In addition, $\zeta_{d}=\{d(\Omega_{E})\widetilde{\cup} \Omega_{E} : \Omega_{E}\in \mathcal{S}_{E}(X)\}$.
\end{proposition}

\begin{proof}
We will verify that the family $\zeta_{d}$ is a soft convex structure on $X$:\\
(1) Obviously $\Phi_{E}\in\zeta_{d}$, because $\Phi_{E}\in \mathcal{S}_{E}(X)$ and $d(\Phi_{E})= \Phi_{E}$ by the soft normalization law. Since $X_{E}\in\mathcal{S}_{E}(X)$ and $d(X_{E})\widetilde{\subseteq} X_{E}$, also $X_{E}\in \zeta_{d}$.\\
(2) Suppose that $\{(\Omega_{i},E): i\in I\}\subseteq\zeta_{d}$ and let $(\Omega,E)=\displaystyle\widetilde{\bigcap}_{i\in I} (\Omega_{i},E)$. Then, $(\Omega_{i},E)\in \mathcal{S}_{E}(X)$ and $d(\Omega_{i},E)\widetilde{\subseteq} (\Omega_{i},E)$ for all $i\in I$. Observe that $(\Omega,E)\in \mathcal{S}_{E}(X)$ and $(\Omega,E)=\displaystyle\widetilde{\bigcap}_{i\in I} (\Omega_{i},E)\widetilde{\subseteq} (\Omega_{j},E)$ for all $j\in I$. By the soft monotone law, we have $d(\Omega,E)\widetilde{\subseteq} d(\Omega_{j},E)\widetilde{\subseteq} (\Omega_{j},E)$ for all $j\in I$. Thus, $d(\Omega,E)\widetilde{\subseteq} \displaystyle\widetilde{\bigcap}_{i\in I} (\Omega_{i},E)=(\Omega,E)$. Therefore, $(\Omega,E)\in \zeta_{d}$.\\
(3) Suppose that $\{(\Omega_{i},E): i\in I\}\subseteq\zeta_{d}$ is a soft upward directed  family and let $(\Omega,E)=\displaystyle\widetilde{\bigcup}_{i\in I} (\Omega_{i},E)$. Then, $(\Omega_{i},E)\in \mathcal{S}_{E}(X)$ and $d(\Omega_{i},E)\widetilde{\subseteq} (\Omega_{i},E)$ for all $i\in I$. Thus, $(\Omega,E)=\displaystyle\widetilde{\bigcup}_{i\in I} (\Omega_{i},E)\in\mathcal{S}_{E}(X)$ and by the soft upward directed additive law
\begin{eqnarray*}
d(\Omega,E)=d\left(\widetilde{\bigcup}_{i\in I} (\Omega_{i},E)\right)=\displaystyle\widetilde{\bigcup}_{i\in I} d(\Omega_{i},E)\widetilde{\subseteq} \widetilde{\bigcup}_{i\in I} (\Omega_{i},E)=(\Omega,E).
\end{eqnarray*}
Therefore, $(\Omega,E)\in \zeta_{d}$.

Certainly $\zeta_{d}$ is a soft convex structure on $X$. Now, we will show that $\zeta_{d}=\{d(\Omega_{E})\widetilde{\cup} \Omega_{E} : \Omega_{E}\in \mathcal{S}_{E}(X)\}$. For convenience we will denote $\zeta^{\ast}_{d}=\{d(\Omega_{E})\widetilde{\cup} \Omega_{E} : \Omega_{E}\in \mathcal{S}_{E}(X)\}$.
By the soft idempotent law, we have $d(\Omega_{E}\widetilde{\cup} d(\Omega_{E}))\widetilde{\subseteq} \Omega_{E}\widetilde{\cup} d(\Omega_{E})$ for all $\Omega_{E}\in\mathcal{S}_{E}(X)$, which implies that $\zeta^{*}_{d}\subseteq\zeta_{d}$. Also, if $\Omega_{E}\in\zeta_{d}$, then $d(\Omega_{E})\widetilde{\subseteq} \Omega_{E}$ and so, $\Omega_{E}=\Omega_{E}\widetilde{\cup} d(\Omega_{E})\in\zeta^{*}_{d}$. Consequently, $\zeta_{d}=\zeta^{*}_{d}$.	
\end{proof}

\begin{proposition}\label{Prop5.2}
Let $d$ be a soft c-derived operator on $X$, $\zeta_{d}\subseteq\mathcal{S}_{E}(X)$ be the soft convex structure induced by $d$ and $co_{d}$ be the soft convex hull in $\zeta_{d}$. Then, $co_{d}(\Omega_{E})=d(\Omega_{E})\widetilde{\cup} \Omega_{E}$ for all  $\Omega_{E}\in\mathcal{S}_{E}(X)$.	
\end{proposition}

\begin{proof}
Since $\Omega_{E}\widetilde{\subseteq} d(\Omega_{E})\cup \Omega_{E}$ and $d(\Omega_{E})\widetilde{\cup} \Omega_{E}\in\zeta_{d}$, by the soft monotone law of $co_{d}$, we have $co_{d}(\Omega_{E})\widetilde{\subseteq} co_{d}(d(\Omega_{E})\widetilde{\cup} \Omega_{E})=d(\Omega_{E})\widetilde{\cup} \Omega_{E}$. On the other hand, if $O_{E}\in\zeta_{d}$  and $\Omega_{E}\widetilde{\subseteq} O_{E}$, then $d(O_{E})\widetilde{\subseteq} O_{E}$ and $d(\Omega_{E})\widetilde{\subseteq} d(O_{E})$, which implies that $d(\Omega_{E})\widetilde{\cup} \Omega_{E} \widetilde{\subseteq} d(O_{E})\widetilde{\cup} O_{E}=O_{E}$. Now, putting $O_{E}=co_{d}(\Omega_{E})$, we get that $d(\Omega_{E})\widetilde{\cup} \Omega_{E}\widetilde{\subseteq} co_{d}(\Omega_{E})$. Therefore, $co_{d}(\Omega_{E})=d(\Omega_{E})\widetilde{\cup} \Omega_{E}$.
\end{proof}

\begin{definition}
Let $(X,d_{X},E)$ and $(Y,d_{Y},E)$ be two soft c-derived spaces. A soft function $f:X\to Y$  is called  \textbf{soft c-derived preserving} (briefly \textbf{SDP}), if
\begin{center}
$f(d_{X}(\Omega_{E}))\widetilde{\subseteq} f(\Omega_{E})\widetilde{\cup} d_{Y}(f(\Omega_{E}))$ for all $\Omega_{E}\in\mathcal{S}_{E}(X)$.
\end{center}
\end{definition}

We will use the notation $f:(X,d_{X},E)\to (Y,d_{Y},E)$ to represent a soft function between two soft c-derived spaces $(X,d_{X},E)$ and $(Y,d_{Y},E)$.

\begin{proposition}
Let $f: (X,d_{X},E)\to(Y,d_{Y},E)$ and $g:(Y,d_{Y},E) \to(Z,d_{Z},E)$ be two soft functions between soft c-derived spaces. If $f$ and $g$ are SDP, then  $g \circ f$  is also SDP.
\end{proposition}

\begin{proof}
Let $\Omega_{E}\in\mathcal{S}_{E}(X)$. Since $f: (X,d_{X},E)\to(Y,d_{Y},E)$ is SDP, we have
\begin{center}
$f(d_{X}(\Omega_{E}))\widetilde{\subseteq} f(\Omega_{E})\widetilde{\cup} d_{Y}(f(\Omega_{E}))$
\end{center}
and so,
\begin{eqnarray*}
(g\circ f)(d_{X}(\Omega_{E}))&=& g(f(d_{X}(\Omega_{E})))\\
&\widetilde{\subseteq}& g(f(\Omega_{E}) \widetilde{\cup} d_{Y}(f(\Omega_{E})))\\
&=& g(f(\Omega_{E}))\widetilde{\cup} g(d_{Y}(f(\Omega_{E})))\\
&=& (g\circ f) (\Omega_{E})\widetilde{\cup} g(d_{Y}(f(\Omega_{E}))).
\end{eqnarray*}
Since $g:(Y,d_{Y},E) \to(Z,d_{Z},E)$ is SDP and $f(\Omega_{E})\in \mathcal{S}_{E}(Y)$, we have
\begin{center}
$g(d_{Y}(f(\Omega_{E}))\widetilde{\subseteq} g(f(\Omega_{E}))\widetilde{\cup} d_{Z}(g(f(\Omega_{E})))= (g\circ f) (\Omega_{E}) \widetilde{\cup} d_{Z}((g \circ f)(\Omega_{E}))$,
\end{center}
which implies that
\begin{eqnarray*}
(g\circ f)(d_{X}(\Omega_{E}))&\widetilde{\subseteq}& (g\circ f) (\Omega_{E})\widetilde{\cup} g(d_{Y}(f(\Omega_{E})))\\
&\widetilde{\subseteq}& (g\circ f) (\Omega_{E}) \widetilde{\cup} d_{Z}((g \circ f)(\Omega_{E})).
\end{eqnarray*}
This shows that $g \circ f$ is SDP.
\end{proof}

\begin{proposition}\label{prop 4.3.8}
Let $(X,d_{X},E)$ and $(Y,d_{Y},E)$ be two c-derived spaces. If $f:(X,d_{X},E)\to(Y,d_{Y},E)$ is SDP, then $f:(X,\zeta_{d_{X}},E)\to (Y,\zeta_{d_{Y}},E)$ is SCP.
\end{proposition}

\begin{proof}
Assume that $f:(X,d_{X},E)\to(Y,d_{Y},E)$ is SDP and let $\Omega_{E}\in\mathcal{S}_{E}(X)$. Then,  $f(d_{X}(\Omega_{E}))\widetilde{\subseteq} f(\Omega_{E})\widetilde{\cup} d_{Y}(f(\Omega_{E}))$ and by Proposition \ref{Prop5.2}, we have $co_{d_{X}}(\Omega_{E})=d_{X}(\Omega_{E})\widetilde{\cup} \Omega_{E}$ and $co_{d_{Y}}(f(\Omega_{E}))=d_{Y}(f(\Omega_{E}))\widetilde{\cup} f(\Omega_{E})$. Thus,
$f(co_{d_{X}}(\Omega_{E}))= f(d_{X}(\Omega_{E})\widetilde{\cup} \Omega_{E})=f(d_{X}(\Omega_{E}))\widetilde{\cup}\break f(\Omega_{E})\widetilde{\subseteq} f(\Omega_{E})\widetilde{\cup} d_{Y}(f(\Omega_{E}))=co_{d_{Y}}(f(\Omega_{E}))$ and by Theorem \ref{teorema 4.2}, we get that $f:(X,\zeta_{d_{X}},E)\to (Y,\zeta_{d_{Y}},E)$ is SCP.
\end{proof}

\section{Soft convex base spaces}
In this section, we present the notions of soft convex bases and soft convex base spaces. Also, we establish the relationship between these notions and soft convex structures.
\begin{definition}
A family $\beta\subseteq\mathcal{S}_{E}(X)$ is a \textbf{soft convex base} (briefly \textbf{soft c-base}) on $X$, if the following conditions are satisfied:
\begin{enumerate}[\upshape(SCB1)]
\item $X_{E}=\displaystyle\widetilde{\bigcup}_{i \in I}\left(\Omega_{i}, E\right)$ for some soft upward directed family $\{(\Omega_{i},E): i\in I\}\subseteq\beta$.
\item For each family $\{(\Omega_{i},E): i\in I\}\subseteq\beta$, there exists a soft upward directed family $\{(O_{j},E): j\in J\}\subseteq\beta$ such that $\displaystyle\widetilde{\bigcap}_{i \in I}\left(\Omega_{i}, E\right)=\displaystyle\widetilde{\bigcup}_{j \in J}\left(O_{j}, E\right)$.
\item If $\{(\Omega_{i},E): i\in I\}\subseteq \mathcal{S}_{E}(X)$ is a soft upward directed family and $\{(O_{ij},E): j\in J_{i}\}\subseteq \beta$ is a soft upward directed family such that $(\Omega_{i},E)=\displaystyle\widetilde{\bigcup}_{j \in J_{i}}\left(O_{ij}, E\right)$ for all $i\in I$, then there exists a soft upward directed family $\{(O_{k},E): k\in K\}\subseteq\beta$ such that $\displaystyle\widetilde{\bigcup}_{i \in I}\left(\Omega_{i}, E\right)=\displaystyle\widetilde{\bigcup}_{k \in K}\left(O_{k}, E\right)$.
\end{enumerate}
\end{definition}

If $\beta$ is a soft c-base on $X$, then the triplet $(X,\beta,E)$ is called a \textbf{soft convex base space} (briefly \textbf{soft c-base space}).

In the following theorem, we will represent by $\mathfrak{U}(\beta)$ the set of all soft upward directed families contained in a soft c-base $\beta$ on $X$.
\begin{theorem}
Let $\beta$ be a soft c-base on $X$. Then, the family $\zeta_{\beta}\subseteq\mathcal{S}_{E}(X)$ defined by
\begin{center}
$\zeta_{\beta}=\left\{\Omega_{E}\in\mathcal{S}_{E}(X): (\Omega,E)=\displaystyle\widetilde{\bigcup}_{i \in I}\left(\Omega_{i}, E\right) \mbox{ for some family } \{(\Omega_{i},E): i\in I\}\in\mathfrak{U}(\beta)\right\}$,
\end{center}
is a soft convex structure on $X$, called the \textbf{soft convex structure generated by $\beta$}.
\end{theorem}

\begin{proof}
We will verify that the family $\zeta_{\beta}$ is a soft convex structure on $X$:\\
(1) Clearly $X_{E}\in \zeta_{\beta}$ by (SCB1). If $\Phi_{E}\notin \zeta_{\beta}$, then $\Phi_{E}\neq \displaystyle\widetilde{\bigcup}_{i \in I}\left(\Omega_{i}, E\right)$ for all soft upward directed family $\{(\Omega_{i},E): i\in I\}\in\mathfrak{U}(\beta)$. In particular, as the empty family   $\{(\Omega_{i},E): i\in \emptyset\}$ of soft sets on $X$ is soft upward directed and is contained in $\beta$, we have $\Phi_{E}\neq \displaystyle\widetilde{\bigcup}_{i \in \emptyset}\left(\Omega_{i}, E\right)$. But this is a contradiction, because in this case $\displaystyle\bigcup_{i \in \emptyset} \Omega_{i}(e)=\emptyset$ for all $e\in E$, which means that $\displaystyle\widetilde{\bigcup}_{i \in \emptyset}\left(\Omega_{i}, E\right)=\Phi_{E}$. Since the assumption that $\Phi_{E}\notin \zeta_{\beta}$ leads us to a contradiction, we conclude that $\Phi_{E}\in \zeta_{\beta}$.\\
(2) Suppose that $\{(O_{i},E): i\in I\}\subseteq\zeta_{\beta}$. Then, $(O_{i},E)\in \mathcal{S}_{E}(X)$ and there exists a soft upward directed family $\{(\Omega_{ij},E): j\in J_{i}\}\subseteq\beta$ such that $(O_{i},E)=\displaystyle\widetilde{\bigcup}_{j\in J_{i}}\left(\Omega_{ij}, E\right)$ for all $i\in I$. Thus,
\begin{center}
$\displaystyle\widetilde{\bigcap}_{i\in I} (O_{i},E)=\widetilde{\bigcap}_{i\in I}\widetilde{ \bigcup}_{j\in J_{i}} (\Omega_{ij},E)=\widetilde{\bigcup}_{f\in\prod_{i\in I} J_{i}} \widetilde{\bigcap}_{i\in I} (\Omega_{if(i)},E)$.
\end{center}
By (SCB2), it follows that for all $f\in\prod_{i\in I}J_{i}$, there exists a soft upward directed family $\{(\Omega_{ik},E): k\in K_{i}\}\subseteq\beta$ such that $\displaystyle\widetilde{\bigcap}_{i\in I} (\Omega_{if(i)},E)=\displaystyle\widetilde{\bigcup}_{k\in K_{i}} (\Omega_{ik},E)$. We affirm that
 $\left\{\displaystyle\widetilde{\bigcap}_{i\in I} (\Omega_{i f(i)},E): f\in \prod_{i\in I} J_i\right\}$ is a soft upward directed family of soft sets on $X$. Indeed, if $f, g \in \displaystyle\prod_{i\in I} J_i$, then $f(i)$ and $g(i)$ belong to $J_{i}$, $(\Omega_{i f(i)},E)$ and $(\Omega_{i g(i)},E)$ belong to  $\{(\Omega_{ij},E): j\in J_{i}\}$ for all $i\in I$. As $\{(\Omega_{ij},E): j\in J_{i}\}$ is a soft upward directed family, there exists $(\Omega_{ij_{i}},E)\in \{(\Omega_{ij},E): j\in J_{i}\}$  such that $(\Omega_{i f(i)},E)\widetilde{\subseteq} (\Omega_{ij_{i}},E)$ and $(\Omega_{i g(i)},E)\widetilde{\subseteq} (\Omega_{ij_{i}},E)$. Let us define the function $h:I\to\displaystyle\bigcup_{i\in I} J_{i}$ by $h(i)=ij_{i}$ for all $i\in I$. Thus, $h\in \displaystyle\prod_{i\in I} J_{i}$, $\displaystyle\widetilde{\bigcup}_{i\in I} (\Omega_{i f(i)},E)\widetilde{\subseteq} \widetilde{\bigcup}_{i\in I} (\Omega_{i h(i)},E)$ and  $\displaystyle\widetilde{\bigcup}_{i\in I} (\Omega_{i g(i)},E)\widetilde{\subseteq} \widetilde{\bigcup}_{i\in I} (\Omega_{i h(i)},E)$. This shows that what we had affirmed is true. Now, by (SCB3), there exists a soft upward directed family $\beta^{*}\subseteq\beta$ such that
\begin{center}
$\displaystyle\widetilde{\bigcap}_{i\in I} (O_{i},E)=\widetilde{\bigcup}_{f\in\prod_{i\in I} J_{i}} \widetilde{\bigcap}_{i\in I} (\Omega_{if(i)},E)=\widetilde{\bigcup}_{(\Theta,E)\in\beta^{*}} (\Theta,E)$,
\end{center}
which implies that $\displaystyle\widetilde{\bigcap}_{i\in I} (O_{i},E)\in \zeta_{\beta}$.\\
(3) Suppose that $\{(O_{i},E): i\in I\}\subseteq\zeta_{\beta}$ is a soft upward directed  family. Then, for all $i\in I$, $(O_{i},E)\in \mathcal{S}_{E}(X)$ and there exists a soft upward directed family $\{(\Omega_{ij},E): j\in J_{i}\}\subseteq\beta$ such that $(O_{i},E)=\displaystyle\widetilde{\bigcup}_{j\in J_{i}}\left(\Omega_{ij}, E\right)$. Thus,  by (SCB3), there exists a soft upward directed family $\{(\Omega_{k},E): k\in K\}\subseteq\beta$ such that
\begin{center}
$\displaystyle\widetilde{\bigcup}_{i\in I} (O_{i},E)=\widetilde{\bigcup}_{k\in K} (\Omega_{k},E)$.
\end{center}
Therefore, $\displaystyle\widetilde{\bigcup}_{i\in I} (O_{i},E)\in \zeta_{\beta}$.

From the above, we conclude that  $\zeta_{\beta}$ is a soft convex structure on $X$.
\end{proof}

\begin{definition}
Let $(X,\beta_{X},E)$ and $(Y,\beta_{Y},E)$ be two soft c-base spaces. A soft function $f:X\to Y$  is called  \textbf{soft c-base preserving} (briefly \textbf{SBP}), if $\Omega_{E}\in\beta_{Y}$ implies $f^{-1}(\Omega_{E})\in\beta_{X}$.
\end{definition}

We will use the notation $f:(X,\beta_{X},E)\to (Y,\beta_{Y},E)$ to represent a soft function between two soft c-base spaces $(X,\beta_{X},E)$ and $(Y,\beta_{Y},E)$.

\begin{proposition}
Let $f: (X,\beta_{X},E)\to(Y,\beta_{Y},E)$ and $g:(Y,\beta_{Y},E) \to(Z,\beta_{Z},E)$ be two soft functions between soft c-base spaces. If $f$ and $g$ are SBP, then  $g \circ f$  is also SBP.
\end{proposition}

\begin{proposition}\label{prop 4.3.8}
Let $(X,\beta_{X},E)$ and $(Y,\beta_{Y},E)$ be two soft c-base spaces. If $f:(X,\beta_{X},E)\to(Y,\beta_{Y},E)$ is SBP, then $f:(X,\zeta_{\beta_{X}},E)\to (Y,\zeta_{\beta_{Y}},E)$ is SCP.
\end{proposition}

\begin{proof}
Let $(\Omega,E)\in \zeta_{\beta_{Y}}$. Then, there exists a soft upward directed family $\left\{(\Omega_{i},E): i\in I\right\}\subseteq \beta_{Y}$ such that $(\Omega,E)=\displaystyle\widetilde{\bigcup}_{i \in I} (\Omega_{i},E)$. Thus, $\left\{f^{-1}(\Omega_{i},E): i\in I\right\}\subseteq \beta_{X}$ is a soft upward directed family and $f^{-1}(\Omega,E)=f^{-1}\left(\displaystyle\widetilde{\bigcup}_{i \in I} (\Omega_{i},E)\right)=\displaystyle\widetilde{\bigcup}_{i \in I} f^{-1}(\Omega_{i},E)$, which implies that $f^{-1}(\Omega,E)\in \zeta_{\beta_{X}}$. Therefore, $f:(X,\zeta_{\beta_{X}},E)\to (Y,\zeta_{\beta_{Y}},E)$ is SCP.
\end{proof}

\end{document}